\documentclass[11pt,leqno]{amsproc}

\usepackage{hyperref}

\newtheorem{theorem}{Theorem}[section]
\newtheorem{lemma}[theorem]{Lemma}

\theoremstyle{definition}

\title{On verbal subgroups in finite and profinite groups}%Paper title
\author{Cristina Acciarri and Pavel Shumyatsky}%Authors names
\address[C.~Acciarri]{Department of Mathematics, University of Brasilia, Brasilia-DF, 70910-900 Brazil}\email{acciarricristina@yahoo.it}
\address[P.~Shumyatsky]{Department of Mathematics, University of Brasilia, Brasilia-DF, 70910-900 Brazil}\email{pavel@unb.br}



\thanks{This research was supported by CNPq-Brazil.}

\begin{document}

\begin{abstract}
Let $w$ be a multilinear commutator word. In the present paper we describe recent results that show that if $G$ is a profinite group in which all $w$-values are contained in a union of finitely (or in some cases countably) many subgroups with a prescribed property, then the verbal subgroup $w(G)$ has the same property as well. In particular, we show this in the case where the subgroups are periodic or of finite rank.
\end{abstract}

\maketitle

\section*{Introduction}
Let $w$ be a group-word, and let $G$ be a group. The verbal subgroup $w(G)$ of $G$ determined by $w$ is the subgroup generated by the set of all values $w(g_1,\ldots,g_n)$, where $g_1,\ldots,g_n$
are elements of $G$. Most of the words considered in this paper are {\it multilinear commutators}, also known under the name of {\it outer commutator words}.  These are words that have a form of a multilinear Lie monomial, i.e., they are constructed by nesting commutators using always different indeterminates.   For example the word  $$[[x_1,x_2],[y_1,y_2,y_5],z]$$ is a multilinear commutator  while the Engel word  $$[x,y,y,y]$$ is not. 

 An important family of multilinear commutators are the lower central words $\gamma_k$,
given by
\[
\gamma_1=x_1,
\qquad
\gamma_k=[\gamma_{k-1},x_k]=[x_1,\ldots,x_k],
\quad
\text{for $k\ge 2$.}
\]
The corresponding verbal subgroups $\gamma_k(G)$ are the terms of the lower central series of $G$.  Another distinguished sequence of multilinear  commutators are the derived words $\delta_k$, on $2^k$ variables, which are defined recursively by 
\[
\delta_0=x_1,
\quad
\delta_k=[\delta_{k-1}(x_1,\ldots,x_{2^{k-1}}),\delta_{k-1}(x_{2^{k-1}+1},\ldots,x_{2^k})],\quad
\text{for $k\ge 1$.}
\] 
The verbal subgroup that corresponds to the word $\delta_k$ is the familiar $k$th derived subgroup of $G$ usually denoted by $G^{(k)}$.

 A word $w$ is  concise  in a class of groups $\mathcal{C}$ if whenever the set of all $w$-values in $G$ is finite, it always follows that $w(G)$ is finite, for every group $G \in \mathcal{C}$. P. Hall asked whether every word is concise, but it was later proved by Ivanov \cite{I} that this problem has a negative solution (see also \cite{O}, p.\ 439). On the other hand, many relevant words are known to be concise. For instance, Turner-Smith \cite{T2} showed that the lower central words $\gamma_k$ and the derived words $\delta_k$ are concise.  In 1974 Wilson \cite{Wil} extended the result to all multilinear commutator words. New proofs of the results on $\gamma_k$ and $\delta_k$ were obtained in \cite{BKS} and for arbitrary multilinear commutators in \cite{FM}. Turner-Smith also proved that every word is concise in the class formed by the residually finite groups all of whose quotients are again residually finite \cite{T2}. In 1967 Merzljakov \cite{Mer} answered P. Hall's question in the affirmative for the class of linear groups. It is an open question, due to Segal \cite{Se}, whether in the class of residually finite groups every word is concise. 

There are several natural ways to look at Hall's question from a
different angle. The circle of problems arising in this context can be
characterized as follows. 

Given a word $w$ and a group $G$, assume that certain restrictions are
imposed on the set of all $w$-values in $G$. How does this influence the structure of the
verbal subgroup $w(G)$?

In the present paper we describe some recent results related to the above question.  We concentrate on results that show that if $G$ is a profinite group in which all $w$-values are contained in a union of finitely (or in some cases countably) many subgroups with a prescribed property, then the verbal subgroup $w(G)$ has the same property as well.  In particularly we show this in the case where the subgroups are periodic or of finite rank.

Throughout the paper  we use the expression ``$(a, b, \ldots)$-bounded''  to mean  ``bounded from above by a function depending only on the parameters $a$, $b$, $\ldots$''.
%%%%%%%%%%%%%%%%%%%%%%%%%%%%%%%%%%%%%%%%%%%%%%%%%%%
\section{On groups in which $w$-values admit a finite covering}

A  covering of a group $G$ is a family $\{S_i\}_{i\in I}$ of subsets of
$G$ such that $G=\bigcup_{i\in I}\,S_i$.
If $\{H_i\}_{i\in I}$ is a covering of $G$ by subgroups, it is natural to ask
what information about $G$ can be deduced from properties of the subgroups $H_i$. In the case where the covering is finite actually quite a lot about the structure of $G$ can be said. The first result in this direction is due to Baer  \cite{Neu1}, who proved that $G$ admits a finite covering by abelian subgroups if and only if it is central-by-finite. The nontrivial part of this assertion is an immediate consequence of a subsequent result of B.H. Neumann \cite{Neu2}: if $\{S_i\}$ is a finite covering of $G$ by cosets of subgroups, then $G$ is also covered by the cosets $S_i$ corresponding to subgroups of finite index in $G$. In other words, we can get rid of the cosets of subgroups of infinite index without losing the covering property.

If the set of all $w$-values in a group $G$ can be covered by finitely many subgroups, one could hope to get some information about the structure of the verbal subgroup $w(G)$.

In this direction we mention the following result that was obtained in \cite{RS}. Let $w$ be either the lower central word $\gamma_k$ or the derived word $\delta_k$. Suppose that $G$ is a group in which  all  $w$-values are contained in a union of finitely many Chernikov subgroups,  then $w(G)$ is Chernikov.  Remind that a group is Chernikov  if and only if  it is a finite extension of a direct sum of finitely many Pr\"ufer groups $C_{p^\infty}$.

Another result of this nature was established in \cite{FS}: If $G$ is a group in which  all commutators are contained in a union of finitely many cyclic
subgroups, then $G'$ is either cyclic or finite. Later G.\ Cutulo and C.\ Nicotera showed that if $G$ is a group in which  all $\gamma_k$-values are contained in a union of finitely many cyclic subgroups, then $\gamma_k(G)$ is finite-by-cyclic. They also showed that $\gamma_k(G)$ can be neither cyclic nor finite \cite{CN}.

In our recent paper \cite{AS} we dealt with profinite groups in which all $w$-values are contained in a union of finitely many subgroups with certain prescribed properties.  A profinite group is a topological group that is isomorphic to an inverse limit of finite groups. The textbooks \cite{riza} and \cite{wilson} provide a good introduction to the theory of profinite groups. In the context of profinite groups all the usual concepts of group theory are interpreted topologically. In particular, by a verbal subgroup of a profinite group corresponding to the word $w$ we mean the closed subgroup generated by  all $w$-values.

 \begin{theorem} 
 \label{thm1}
 Let $w$ be a multilinear commutator word and $G$ a profinite group that has finitely many periodic subgroups $G_1,G_2,\dots,G_s$ whose union contains all $w$-values in $G$. Then $w(G)$ is locally finite.
\end{theorem}

A group is periodic (torsion) if every element of the group has finite order and  a group is called locally finite if each of its finitely generated subgroups is finite. Periodic profinite groups have received a good deal of attention in the past. In particular, using Wilson's reduction theorem \cite{wil}, Zelmanov has been able to prove local finiteness of periodic compact groups \cite{z}.  Earlier Herfort showed that there exist only finitely many primes dividing the orders of elements of a periodic profinite group \cite{herfort}.  It is a long-standing problem whether any periodic profinite group has finite exponent. Remind that a group $G$ has exponent $e$ if $x^e=1$ for all $x\in G$ and $e$ is the least positive integer with that property.

In \cite{AS} we also considered the situation where all $w$-values are contained in a union of finitely many subgroups of finite rank. A group $G$ is said to be of finite rank $r$ if every finitely generated subgroup of $G$ can be generated by $r$ elements. Very often the term ``special rank" is used with the same meaning. However in the present paper there is no danger of confusing different notions so we will just use the term ``rank". We have the following result.

\begin{theorem}
\label{444} 
Let $w$ be a multilinear  commutator  and $G$ a profinite group that has finitely many subgroups $G_1,G_2,\dots,G_s$ whose union contains all $w$-values in $G$. If each of the subgroups  $G_{1},G_{2},\ldots,G_{s}$ is of finite rank, then $w(G)$ has finite rank as well.
\end{theorem}

It follows from the proof  that if under the hypotheses of Theorem  \ref{thm1} the subgroups $G_1,G_2,\dots,G_s$ have finite exponent, then $w(G)$ has finite exponent  as well. It is natural to address the question whether the exponent (rank) of $w(G)$ is bounded in terms of the exponents (ranks) of $G_1,G_2,\dots,G_s$ and $s$.

We have been able to answer the question in the affirmative only in the particular case where $w=\gamma_k$. The case of the exponent was treated using the Lie-theoretic techniques that Zelmanov created in his solution of the restricted Burnside problem \cite{ze1,ze2}.

\begin{theorem}\label{333} Let $e,k,s$ be positive integers and $G$ a profinite group that has  subgroups $G_1,G_2,\dots,G_s$ whose union contains all $\gamma_k$-values in $G$. Suppose that each of the subgroups $G_1,G_2,\dots,G_s$ has finite exponent dividing $e$. Then $\gamma_k(G)$ has finite $(e,k,s)$-bounded exponent.
\end{theorem}
The situation where the subgroups $G_1,G_2,\dots,G_s$ are of finite rank does not require the use of Zelmanov's Lie-theoretic techniques. Instead an important role in the proof of the next theorem is played by the Lubotzky-Mann theory of powerful $p$-groups \cite{LM}.

\begin{theorem}\label{555} Let $k,r,s$ be positive integers and $G$ a profinite group that has  subgroups $G_1,G_2,\dots,G_s$ whose union contains all $\gamma_k$-values in $G$. Suppose that each of the subgroups $G_1,G_2,\dots,G_s$ has finite rank at most $r$. Then $\gamma_k(G)$ has finite $(k,r,s)$-bounded rank.
\end{theorem}

The proofs of the above theorems will be sketched out in Section 4 below.  In the next section we will discuss one specific tool used in the proofs of Theorems \ref{333} and \ref{555} --namely, a version of the Focal Subgroup Theorem.
%%%%%%%%%%%%%%%%%%%%%%%%%%%%%%%%%%%%%%%%%%%%%
\section{Around the focal subgroup theorem}

The Focal Subgroup Theorem \cite[Theorem 7.3.4]{Go} says that if $P$ is a Sylow subgroup of a finite group $G$, then $P\cap G'$ is generated by elements of the form $[x,y]\in P$, where $x\in P$ and $y\in G$. In particular, it follows that Sylow subgroups of $G'$ are generated by commutators. Thus, the following question arises.

Let $w$ be a commutator word, $G$ a finite group and $P$ a Sylow $p$-subgroup of $w(G)$.  Is it true that $P$  can be generated by $w$-values lying in $P$?

The above question was considered in \cite{AFS} where  the following result was proved.

\begin{theorem}
\label{focal}
 Let $w$ be a multilinear  commutator, $G$ a finite group and $P$ a Sylow $p$-subgroup of $w(G)$. Then $P$ is generated by powers of $w$-values. 
\end{theorem} 

The theorem depends on the classification of finite simple groups.  It is used in the proof of Theorems \ref{333} and \ref{555}.

Another obvious corollary of the Focal Subgroup Theorem is that if $G$ is a finite group in which every nilpotent subgroup generated by commutators has exponent dividing $e$, then the exponent of $G'$ divides $e$ as well. It is easy to deduce from Theorem \ref{focal} that if every nilpotent subgroup generated by $\gamma_k$-commutators has exponent dividing $e$, then $\gamma_k(G)$ has $(k,e)$-bounded exponent. The latest result in this direction is the following theorem \cite{le}.
\begin{theorem}\label{aust} Let $e$ be a positive integer and $w$ a multilinear commutator. Suppose that $G$ is a finite group in which any nilpotent subgroup generated by $w$-values has exponent dividing $e$. Then the exponent of the corresponding verbal subgroup $w(G)$ is bounded in terms of $e$ and $w$ only.
\end{theorem} 
 The proof of this result uses a number of deep tools.  In particular, it uses the classification of finite simple groups and Zelmanov's solution of the restricted Burnside problem.  It is interesting that the reduction from the general case to the case where $G$ is soluble is somewhat reminiscent of the Hall-Higman's reduction \cite{HH} for the restricted Burnside problem.  In the same time in the case of Theorem \ref{aust} the reduction uses the solution of the restricted Burnside problem.

As a by-product of the proof of Theorem \ref{aust}, it is shown that if $G$ is a finite soluble group in which any nilpotent subgroup generated by $w$-values has exponent dividing $e$, then the Fitting height of $G$ is bounded in terms of $e$ and $w$ only.
%%%%%%%%%%%%%%%%%%%%%%%%%%%%%%%%%%%%%%%%%%%%%%%%%%%%
\section{About proofs}

In the present section we will describe ideas behind the proof of  Theorem  \ref{thm1}.  Details can be found in \cite{AS}.\ As the reader will see the arguments are pretty universal and can be used to obtain other results of this kind. In particular, in the next section we use the same scheme of reasoning to deduce new results on covering of $w$-values in profinite groups.

\begin{proof}[Idea of the proof of Theorem \ref{thm1}] Without explicit references we use Zelmanov's theorem that a periodic profinite group is locally finite \cite{z}. It is a general property of multilinear commutators that for every $w$ there exists $k$ such that every $\delta_k$-commutator is also a $w$-value \cite[Lemma 4.1]{Shu}. For each integer $i=1,\dots,s$ we set 
\begin{equation*}
S_i=\Big\{(x_1,\dots,x_{2^k})\in \underbrace{G\times\dots\times G}_{2^k}\mid \delta_k(x_1,\dots,x_{2^k})\in G_i\Big\}.
\end{equation*} 
The sets $S_i$ are closed in $\underbrace{G\times\dots\times G}_{2^{k}}$ and cover  the group $\underbrace{G\times\dots\times G}_{2^{k}}$.  By Baire's Category Theorem \cite[p.\ 200]{Kell} at least one of these sets contains a non-empty interior.  Hence, there exist an open subgroup $H$ of $G$, elements $a_1,\dots,a_{2^k}$ in $G$ and an integer $j$ such that 
\begin{equation*}
\delta_k(a_1h_1,\dots,a_{2^k}h_{2^k})\in G_j  \text{ for all } h_1,\dots,h_{2^k}\in H.
\end{equation*} 
Thus, all commutators of the form $\delta_k(a_1h_1,\dots,a_{2^k}h_{2^k})$ belong to the given periodic subgroup $G_j$.

Without loss of generality we can assume that the subgroup $H$ is normal. In this case $H$ normalizes the set of all commutators of the form $\delta_k(a_1h_1,\dots,a_{2^k}h_{2^k})$, where $h_1,\dots,h_{2^k}\in H$.  Let $K$ be the subgroup of $G$ generated by all commutators of the form $\delta_k(a_1h_1,\dots,a_{2^k}h_{2^k})$, where $h_1,\dots,h_{2^k}\in H$. Note that  $K\leq G_{j}$. Since the subgroup $G_{j}$ is  locally finite, so is $K$. Let $D=K\cap H$. Then $D$ is a normal locally finite subgroup of $H$ and the normalizer  of $D$ in $G$ has finite index. Therefore there are only finitely many conjugates of $D$ in $G$.  Let $D=D_1,D_2, \dots,D_r$ be all these conjugates. All of them are normal in $H$ and so their product $D_{1}D_{2}\ldots D_{r}$ is locally finite.  By passing to the quotient $G/D_1D_2\ldots D_r$  we may assume that $D=1$.  Since $D=K\cap H$ and $H$ has finite index in $G$, it follows that $K$ is finite.  On the other hand,  the normalizer of $K$ has finite index in $G$ and so, by Dicman's Lemma \cite[14.5.7]{Rob} the normal closure, say $L$, of $K$ in $G$ is also finite.  We can pass to the quotient group $G/L$ and assume that $K=1$. In that case we have $\delta_k(a_1h_1,\dots,a_{2^k}h_{2^k})=1$ for all $h_1,\dots,h_{2^k}\in H$.

This is a so-called {\it coset identity}, i.e., the cosets $a_{1}H,\ldots,a_{2^{k}}H$ satisfy the law $\delta_{k}\equiv1$. Note that coset identities often  play important role in the results on profinite groups. For example, the problem whether every compact periodic group has finite exponent can be reduced to the following question: Let $n$ be a positive integer. Suppose that a periodic profinite group $G$ has an element $a$ and an open subgroup $H$ such that $(ah)^n=1$ for every $h\in H$. Does it follow that $H$ has finite exponent? So far the positive answer to the above question is known only in the case where $n$ is prime \cite{Khukhro}.

Our case is much easier, due to the fact that the word $\delta_k$ is multilinear. We prove that the condition $\delta_k(a_1h_1,\dots,a_{2^k}h_{2^k})=1$ for all $h_1,\dots,h_{2^k}\in H$ implies that $H$ is soluble with derived length at most $k$. Thus, the group $G$ is soluble-by-finite and we can use induction on the minimal derived length $d$ of a normal open subgroup of $G$. 

So assume that $H$ has derived length $d$ and let $N=H^{(d-1)}$.  Let $M$ be the subgroup of $N$ generated by all $w$-values contained in $N$.  Then $M=M_1M_{2}\ldots M_s$, where $M_i=M\cap G_i$ for $i=1,\ldots,s$.  It follows that $M$ is locally finite and  has finite exponent. Thus  we can pass to the quotient $G/M$ and  assume that there are no nontrivial $w$-values contained in $N$.  
 
It is a property of multilinear commutators that a normal subgroup containing no $w$-values must centralize $w(G)$. Thus, $[w(G),N]=1$.
Recall that we are using induction on $d$ and $N=H^{(d-1)}$. Since $[w(G),N]=1$, we conclude that $w(G)/Z(w(G))$ is locally finite. For abstract groups we have a locally finite version of Schur's Theorem: if $K/Z(K)$ is locally finite, then $K'$ is locally finite. However it is unclear whether  this  can be used on profinite groups.  Fortunately, we have the additional information that $w(G)$ is soluble-by-finite.

Since $w(G)/Z(w(G))$ is locally finite and soluble-by-finite, it easily follows that $w(G)/Z(w(G))$ has finite exponent. By a profinite version of Mann's theorem \cite{mann} we conclude that the derived group of $w(G)$ has finite exponent. Therefore we can pass to the quotient $G/w(G)'$ and assume that $w(G)$ is abelian.  But now it is clear that $w(G)$ is the product of subgroups $w(G)\cap G_i$ for $i=1,\ldots,s$, and so $w(G)$ is locally finite. 
\end{proof} 

Theorem \ref{444} can be proved following, by and large, similar arguments. One essential difference is that in place of Mann's theorem we apply a result, due to S. Franciosi, F. de Giovanni, and L.A. Kurdachenko \cite[Theorem 2.5]{kurda}  which says that if $K$ is a soluble-by-finite group such that $K/Z(K)$ has finite rank, then $K'$ has finite rank bounded in terms of the derived length of the soluble radical of $K$, its index in $K$ and the rank of $K/Z(K)$. 

In the proofs of Theorems \ref{333} and \ref{555} techniques of quite a different nature are required. Let us focus our attention on Theorem \ref{333}.
\medskip

{\it Let $e,k,s$ be positive integers and $G$ a profinite group that has  subgroups $G_1,G_2,\dots,G_s$ whose union contains all $\gamma_k$-values in $G$. Suppose that each of the subgroups $G_1,G_2,\dots,G_s$ has finite exponent dividing $e$. Then $\gamma_k(G)$ has finite $(e,k,s)$-bounded exponent.}
\medskip

Note that without loss of generality  $G$ can be assumed finite. The following elementary lemma plays a crucial role in the proof.
\begin{lemma}
\label{lemma product}
 Let $G$ be a nilpotent group and suppose that  $G$  is  generated by a commutator-closed subset $X$ which is contained in a union of finitely many subgroups $G_1,G_2,\dots,G_s$. Then $G$ can be written as the product  \break$G=G_1G_2\ldots G_s$.
\end{lemma}

\begin{proof}
 Let $K$ be the last nontrivial term of the lower central series of $G$.  Then $K$ is generated by elements of $X$ and so $K=K_1K_2\ldots K_s$, where $K_{i}=K\cap G_{i}$ for $i=1,\ldots,s$.  Arguing by induction on the nilpotency class of $G$ assume $$G=G_1G_2\ldots G_sK_1K_2\ldots K_s.$$  Since all $K_i$ are central, we can move them  to the left in the above product and the lemma follows.
 \end{proof}

 \begin{proof}[Idea of the proof of Theorem \ref{333}] Let $P$ be  a Sylow $p$-subgroup of $\gamma_k(G)$. It is sufficient to show that the exponent of $P$ is $(e,k,s)$-bounded. By Theorem \ref{focal} we know that $P$ is generated by powers of $\gamma_k$-values -- elements of order dividing $e$. Therefore it is sufficient to show that the exponent of $\gamma_k(P)$ is $(e,k,s)$-bounded. Let $Q=\gamma_k(P)$. Then, by Lemma \ref{lemma product},  $Q$ is the product $Q_1Q_2\ldots Q_s$, where $Q_i=G_i\cap Q$ for $i=1,\ldots,s$. Using Zelmanov's Lie-theoretic techniques one can show that $Q$ has $(e,k,s)$-bounded exponent. Hence $P$ has $(e,k,s)$-bounded exponent.  This holds for every Sylow subgroup of $\gamma_k(G)$ and so  the result follows.
 \end{proof}
%%%%%%%%%%%%%%%%%%%%%%%%%%%%%%%%%%%%%%%%%%%%%%%%%%%%%%%%%%
\section{On countable coverings of $w$-values in profinite groups}

The last section of the paper deals with the situation where $w$-values are covered by countably many subgroups with prescribed properties. 

\begin{theorem}\label{cm1} Let  $w$ be a multilinear  commutator word and $G$  a profinite group having countably many soluble subgroups whose union contains all $w$-values. Then $w(G)$ is soluble-by-finite. 
\end{theorem}

\begin{proof}
Let $G_{1},G_{2}, \ldots$ be the countably many soluble subgroups  whose union contains all $w$-values. 
Since $w$ is a multilinear commutator there exists an integer $k$ such that every $\delta_k$-commutator is also a $w$-value \cite[Lemma 4.1]{Shu}. For each integer $i\geq1$ we set 
\begin{equation*}
S_i=\Big\{(x_1,\dots,x_{2^k})\in \underbrace{G\times\dots\times G}_{2^k}\mid \delta_k(x_1,\dots,x_{2^k})\in G_i\Big\}.
\end{equation*} 
The sets $S_i$ are closed in $\underbrace{G\times\dots\times G}_{2^{k}}$ and cover  the group $\underbrace{G\times\dots\times G}_{2^{k}}$.  By Baire's Category Theorem  at least one of these sets contains a non-empty interior.  Hence, there exist an open subgroup $H$ of $G$, elements $a_1,\dots,a_{2^k}$ in $G$ and an integer $j$ such that 
\begin{equation*}
\delta_k(a_1h_1,\dots,a_{2^k}h_{2^k})\in G_j  \text{ for all } h_1,\dots,h_{2^k}\in H.
\end{equation*} 

Without loss of generality we can assume that the subgroup $H$ is normal. In this case $H$ normalizes the set of all commutators of the form $\delta_k(a_1h_1,\dots,a_{2^k}h_{2^k})$, where $h_1,\dots,h_{2^k}\in H$.  Let $K$ be the subgroup of $G$ generated by all commutators of the form $\delta_k(a_1h_1,\dots,a_{2^k}h_{2^k})$, where $h_1,\dots,h_{2^k}\in H$.  Note that  $K\leq G_{j}$. Since the subgroup $G_{j}$ is  soluble, so is  $K$.  Let $D=K\cap H$. Then $D$ is a normal soluble subgroup of $H$ and the normalizer  of $D$ in $G$ has finite index. Therefore there are only finitely many conjugates of $D$ in $G$.  Let $D=D_1,D_2, \dots,D_r$ be all these conjugates. All of them are normal in $H$ and so their product $D_{1}D_{2}\ldots D_{r}$ is soluble.  By passing to the quotient $G/D_1D_2\ldots D_r$  we may assume that $D=1$.  Since $D=K\cap H$ and $H$ has finite index in $G$, it follows that $K$ is finite. 

On the other hand,  the normalizer of $K$ has finite index in $G$ and so the normal closure, say $L$, of $K$ in $G$ is  finite.  We can pass to the quotient group $G/L$ and assume that $K=1$. In that case we have $\delta_k(a_1h_1,\dots,a_{2^k}h_{2^k})=1$ for all $h_1,\dots,h_{2^k}\in H$. It follows that the  group $H$  is soluble-by-finite-by-soluble, so in particular it is soluble-by-finite.  Since $H$ has finite index in $G$ we conclude that $G$ is soluble-by-finite, as desired.  
\end{proof}

Our attempts to treat the groups in which $w$ values are contained in a union of countably many subgroups that are periodic or of finite rank so far were successful only in the case where $w=[x,y]$.

\begin{theorem}\label{cm2} Let $G$ be a profinite group having countably many periodic subgroups whose union contains all commutators $[x,y]$ of $G$. Then $G'$ is locally finite.
\end{theorem}

\begin{theorem}\label{cm3} Let $G$ be a profinite group having countably many subgroups of finite rank whose union contains all commutators $[x,y]$ of $G$. Then $G'$ is of finite rank.
\end{theorem}

We will now describe in detail the proof of Theorem \ref{cm2}. Theorem \ref{cm3} can be proved in a somewhat similar way.
\begin{proof}[Proof of Theorem \ref{cm2}]
Let $G_{1},G_{2}, \ldots$ be the countably many periodic subgroups  whose union contains all commutators $[x,y]$ of $G$. For each integer $i\geq 1$ we set 
\begin{equation*}
S_i=\Big\{(x,y)\in G\times G \mid [x,y] \in G_i\Big\}.
\end{equation*} 
The sets $S_i$ are closed in $G\times G$ and cover the whole of $G\times G$. By Baire's Category Theorem  at least one of these sets contains a non-empty interior. Hence, there exist an open subgroup $H$ of $G$, elements $a,b$ in $G$ and an integer $i$ such that 
\begin{equation*}
[ah_1,bh_{2}]\in G_i  \text{ for all } h_1,h_{2}\in H.
\end{equation*}
Without loss of generality we can assume that the subgroup $H$ is normal. In this case $H$ normalizes the set of all commutators  $[ah_1,bh_{2}]$, where $h_1,h_{2}\in H$. Let $K$ be the subgroup of $G$ generated by all commutators of the form $[ah_1,bh_{2}]$, where $h_1,h_{2}\in H$. Note that  $K\leq G_{i}$. Since the subgroup $G_{i}$ is  locally finite, so is  $K$.  Let $K_{0}=K\cap H$. Then $K_{0}$ is a normal locally finite subgroup of $H$ and the normalizer  of $K_{0}$ in $G$ has finite index. Therefore there are only finitely many conjugates of $K_{0}$ in $G$. Let $K_0,K_1,\dots,K_r$ be the conjugates. All of them are normal in $H$ and so their product $K_{0}K_{1}\ldots K_{r}$ is locally finite. By passing to the quotient $G/K_0K_1\ldots K_r$  we may assume that $K_{0}=1$.  Since $K_{0}=K\cap H$ and $H$ has finite index in $G$, it follows that $K$ is finite.  On the other hand,  the normalizer of $K$ has finite index in $G$ and so the normal closure, say $L$, of $K$ in $G$ is also finite. We can pass to the quotient group $G/L$ and assume that $K=1$. In that case we have $[ah_1,bh_{2}]=1$ for all $h_1,h_{2}\in H$. As we have seen in the proof of Theorem \ref{thm1}  this implies that the subgroup $H$ is abelian.

Since $H$ is open, we can choose finitely many elements $a_{1}, a_{2},\ldots, a_{s}\in G$ such that $G=\langle H, a_{1}, a_{2},\ldots, a_{s} \rangle$. Then $[H,G]=\prod_{1\leq i\leq s}[H,a_{i}].$ Because $H$ is abelian, every element of the subgroup $[H,a_{i}]$ is a commutator. Again by Baire's Category Theorem  $[H,a_{i}]$ contains an open subgroup $M$ and an element $c$ such that $cM$ is contained in $G_{j}$, for some index $j$. So $M$ is locally finite  and it follows that $[H,a_{i}]$ is locally finite as well. Since we can repeat the same argument for every $i=1,\ldots,s$, we conclude that $[H,G]$ is locally finite. We  now pass to the quotient $G/[H,G]$ and assume that $[H,G]=1$. Then $H\leq Z(G)$ and by Schur's Theorem $G'$ is finite. The proof is complete.
\end{proof}

%%%%%%%%%%%%%%%%%%%%%%%%%%%%%%%%%%%%%%%%%%%%%%%%%%%%
%%%%%%%%%%%%%%%%%%% BIBLIOGRAPHY %%%%%%%%%%%%%%%%%%%%%%%%%%%%

\end{document}